\documentclass[12pt,reqno]{amsart}

\usepackage{amssymb}
\usepackage{amscd}
\usepackage{amsfonts}
\usepackage{setspace}
\usepackage{mathrsfs}
\usepackage{version}
\usepackage{xcolor}

\usepackage{graphicx}

\newtheorem{theorem}{Theorem}[section]
\newtheorem{lemma}[theorem]{Lemma}
\newtheorem{proposition}[theorem]{Proposition}

\theoremstyle{definition}

\newtheorem{definition}{Definition}[section]

\renewcommand{\leq}{\leqslant}
\renewcommand{\geq}{\geqslant}

\def\R{\mathbb{R}}
\def\C{\mathbb{C}}
\def\Z{\mathbb{Z}}
\def\E{\mathbb{E}}
\def\P{\mathbb{P}}

\def\N{\mathbb{N}}

\def\eps{\varepsilon}

\newcommand{\md}[1]{\ensuremath{(\operatorname{mod}\, #1)}}
\newcommand{\mdsub}[1]{\ensuremath{(\mbox{\scriptsize mod}\, #1)}}

\numberwithin{equation}{section}

\begin{document}

\title[Correlation sequences not approximable by nilsequences]{Multiple correlation sequences not approximable by nilsequences}


\author[Bri\"et]{Jop Bri\"et}
\address{Centrum Wiskunde \& Informatica (CWI) \\ Science Park 123 \\ 1098 XG Amsterdam \\The Netherlands}
\email{j.briet@cwi.nl}

\author[Green]{Ben Green}
\address{Mathematical Institute \\ Andrew Wiles Building \\ Radcliffe Observatory Quarter \\ Woodstock Rd \\ Oxford OX2 6QW}
\email{ben.green@maths.ox.ac.uk}

\thanks{The first author is supported by the Gravitation grant NETWORKS-024.002.003 from the Dutch Research Council (NWO).
The second author is supported by a Simons Investigator Award and is grateful to the Simons Foundation for their support.}


\begin{abstract}
We show that there is a measure-preserving system $(X,\mathscr{B}, \mu, T)$ together with functions $F_0, F_1, F_2 \in L^{\infty}(\mu)$ such that the correlation sequence $C_{F_0, F_1, F_2}(n) = \int_X F_0 \cdot T^n F_1 \cdot T^{2n} F_2 d\mu$ is not an approximate integral combination of $2$-step nilsequences.
\end{abstract}
\maketitle

\def\sr{\mathscr{S}}

\section{Introduction}

Let $(X,\mathscr{B}, \mu, T)$ be a measure-preserving system, and let $F_0, F_1,\dots,$ $F_k \in L^{\infty}(\mu)$. Motivated in large part by applications in combinatorics and in particular to questions about arithmetic progressions, there has been much interest in \emph{multiple correlation sequences}
\[ C_{F_0,\dots, F_k}(n) := \int_X F_0 \cdot T^{n} F_1 \cdots T^{kn} F_k d\mu.\]
In fact, much more general types of correlation sequences in which the powers $T, T^2,\dots, T^k$ appearing here are replaced by measure-preserving maps $T_1,\dots, T_k$ have been studied, but here we restrict attention here to this special form. 

In the case $k = 1$,  there is a very satisfactory spectral theory of such sequences and indeed one has
\begin{equation}\label{spectral} C_{F_0, F_1}(n) = \int^1_0 e^{-2\pi i n t} d\sigma(t)\end{equation} for some complex Borel measure $\sigma$ of bounded total variation. This follows from the Herglotz theorem on positive definite sequences (which applies directly in the case $F_0 = \overline{F_1}$) and a depolarization identity.

It is natural to ask to what extent this generalises to $k \geq 2$. In the words of Frantzikinakis \cite{nikos-open-problems},

``Finding a formula analogous to \eqref{spectral}, with the multiple correlation sequences  in
place of the single correlation sequences, is a problem of fundamental importance which
has been in the mind of experts for several years. A satisfactory solution is going to
give us new insights and significantly improve our ability to deal with multiple ergodic
averages.''

A result of Bergelson, Host and Kra \cite{bhk} describes the structure of multiple correlation sequences up to an error in $\ell^1$ or $\ell^2$. To state their result, we need to recall the notion of a nilsequence.

\begin{definition}[Nilsequence]
Let $k \geq 1$ be an integer. A $k$-step nilsequence is a sequence $(\phi(g^n x_0))_{n \in \Z}$. Here, $\phi : G \rightarrow \C$ is a continuous function satisfying the automorphy\footnote{Essentially equivalently, $\phi$ is a function on the \emph{nilmanifold} $G/\Gamma$.} condition $\phi(x\gamma) = \phi(x)$ for all $x \in G$ and all $\gamma \in \Gamma$, where $G$ is a simply-connected $k$-step nilpotent Lie group with discrete and cocompact subgroup $\Gamma$, and $g,x_0$ are fixed elements of $G$.\end{definition}

A careful discussion of this notion may be found in many places, for instance \cite{bhk}.  The following result is \cite[Theorem 1.9]{bhk}.

\begin{theorem}\label{bhk-thm}
Suppose that $(X,\mathscr{B}, \mu, T)$ is a measure-preserving system and that $F_0, F_1,\dots, F_k \in L^{\infty}(\mu)$. Suppose that $\Vert F_i \Vert_{\infty} \leq 1$. Then we have a decomposition
\[ C_{F_0,F_1,\dots, F_k}(n) = a(n) + b(n),\]
where $a(n)$ is a uniform limit of $k$-step nilsequences with $\Vert a \Vert_{\infty} \leq 1$, and $b$ is small in the sense that 
\[ \lim_{|I| \rightarrow \infty} \frac{1}{|I|}\sum_{n \in I} |b(n)| = 0\] as $I$ ranges over all subintervals of $\N$.
\end{theorem}

For applications involving the behaviour of correlation sequences at a sparse sequence of $n$, the error term here is too big. Frantzikinakis \cite[Problem 1]{nikos-open-problems} has suggested, in the context of seeking a generalisation of \eqref{spectral}, that a variant of Theorem \ref{bhk-thm} should hold with an $\ell^{\infty}$ error term. Note that in \eqref{spectral}, we have not just one nilsequence $(e^{2\pi i n t})_{n \in \N}$, but an integral combination of (1-step) nilsequences. Frantzikinakis's formulation generalises this concept to higher-step nilsequences.

\begin{definition}
An integral combination of $k$-step nilsequences is a sequence of the form
\[ a(n) = \int_M a_m(n) d\sigma(m).\] Here, $M$ is a compact metric space, $\sigma$ is a complex Borel measure of bounded variation, and the $a_m$ are $k$-step nilsequences, and with the map $m \mapsto a_m(n)$ being measurable for each $n$. 
\end{definition}

Our main theorem states that, even in the case $k = 2$, one cannot hope for a version of Theorem \ref{bhk-thm} in which the error $b$ is small in $\ell^{\infty}$, even if one allows $a$ to be an integral combination of nilsequences.

\begin{theorem}\label{mainthm}
There is a measure-preserving system $(X,\mathscr{B}, \mu, T)$, functions $F_0, F_1, F_2 \in L^{\infty}(\mu)$ and an $\eps > 0$ such that the correlation sequence
\[ C_{F_0, F_1, F_2}(n) := \int_X F_0 \cdot T^n F_1 \cdot T^{2n} F_2 d\mu\] cannot be written as $a(n) + b(n)$, where $\Vert b \Vert_{\infty} \leq \eps$ and $a$ is an integral combination of $2$-step nilsequences.
\end{theorem}

This theorem casts some serious doubt on the existence of a formula generalising \eqref{spectral}. 

Theorem \ref{mainthm} does not provide a negative answer to \cite[Problem 1]{nikos-open-problems}, because Frantzikinakis allows the automorphic functions $\phi$ in the definition of a nilsequence to be merely Riemann-integrable, rather than continuous. He calls these \emph{generalised} nilsequences. An explanation of why our construction does not allow one to establish an analogue of Theorem \ref{mainthm} for generalised nilsequences is given in Appendix \ref{app-A}. Note, however, that the Riemann-integrable functions $\phi$ appearing in Appendix \ref{app-A} are very singular and we certainly do not expect that the corresponding generalised nilsequences have any important role to play in the theory.

One reason for considering Riemann-integrable functions rather than just continuous ones is that there is a somewhat natural and well-studied class of nilsequences in which $\phi$ is not continuous, namely the bracket polynomial phases \cite{bl-bracket}. In this case, the corresponding $\phi$ have only mild discontinuities, and our argument adapts easily to show that Theorem \ref{mainthm} remains true even if one allows $a$ to be an integral combination of this more general class of nilsequences. We sketch the argument at the end of Section \ref{sec3-ent}. 

A key motivation for Frantzikinakis in formulating \cite[Problem 1]{nikos-open-problems} was that it provides a potential route to understanding Szemer\'edi's theorem with common difference in a sparse random set, a problem for which our current understanding is extremely incomplete for progressions of length 3 or longer (see \cite{jop-random-szem} for recent progress). Whilst Theorem~\ref{mainthm} seems to rule this out as a viable strategy, our example unfortunately does not give any new information about Szemer\'edi's theorem with common differences from a random set, which remains a tantalising open problem.

\emph{Notation.} Our notation is standard. We will occasionally write $\E_{x \in A}$ for $\frac{1}{|A|}\sum_{x \in A}$, where $A$ is a finite set. We write $[N] = \{1,2,\dots, N\}$ as usual, and sometimes we will write $[0,N-1] = \{0,1,2,\dots, N-1\}$. For real $t$, we write $e(t) = e^{2\pi i t}$.

\emph{Acknowledgements.} JB would like to thank Xuancheng Shao for helpful discussions
and pointers to the literature. The authors would like to thank Bryna Kra for pointing them to a reference for Proposition \ref{prop1.1}, and Nikos Frantzikinakis for helpful comments on the first draft of the paper.

\section{Outline of the argument}

Our argument is part deterministic and part random. It is random in the sense that we do not explicitly construct a system $(X,\mathscr{B}, \mu, T)$ and functions $F_0, F_1, F_2$ for which the correlation sequence $C_{F_0, F_1, F_2}(n)$ is not approximable by an integral combination of nilsequences, but rather we show there are too many possibilities for the correlation functions $C_{F_0, F_1, F_2}(n)$ for this to be so.

To do this, we first explicitly construct a certain infinite sequence $\sr \subset \N$ whose growth is slower than exponential in the sense that 
\begin{equation}\label{s-sparse} \lim_{N \rightarrow \infty} \frac{|\sr[N]|}{\log N} = \infty,\end{equation}
where $\sr[N] := \# \{n \in \sr : n \leq N\}$. 

We show that for any choice of function $\eta : \sr \rightarrow \{1, -\frac{1}{3}\}$ there is a system $(X,\mathscr{B}, \mu, T)$ and functions $F_0, F_1, F_2$ such that $C_{F_0 ,F_1, F_2}(n) = \eta(n)$ for $n \in \sr$.

For a random choice of $\eta$, such a function will almost surely not be approximable by an integral combination of nilsequences. We give the details of this deduction, which uses nothing about $\sr$ other than the growth property \eqref{s-sparse}, in Proposition \ref{lem0.1}. 

The heart of the argument, then, is the construction of the system $(X,\mathscr{B}, \mu, T)$ and the functions $F_0, F_1, F_2$, given $\eta : \sr \rightarrow \{1, -\frac{1}{3}\}$. This is assembled from a sequence of finitary examples, via a (well-known) variant of Furstenberg's correspondence principle, and here the specific nature of $\sr$ is critical.

The idea behind the construction of these finitary examples ultimately comes from coding theory, and in particular a construction of Yekhanin \cite{yekhanin}. We will only need the most basic form of these ideas; for instance, we can replace all the finite-field theory in Yekhanin's work with the simple observation that the function $\psi : \Z \rightarrow \{-1,1\}$ defined by $\psi(0) = 1$, $\psi(1) = \psi(2) = -1$, and periodic mod $3$ has the property that \[ \psi(x) \psi(x+d)\psi(x+2d) = \left\{\begin{array}{ll}  \psi(x)   & d \equiv 0 \md{3} \\ 1 & d \neq 0 \md{3} .\end{array} \right.\]
The idea of using Yekhanin's construction to give interesting examples in the additive combinatorics of higher-order correlations first arose in the finite field setting, in joint work of the first author and Labib \cite{jop-labib}. Those ideas have inspired the present work.

\section{Entropy and nilsequences}\label{sec3-ent}

\begin{proposition}\label{lem0.1}
Let $\sr$ be an increasing sequence of natural numbers such that
\begin{equation}\label{log-der} \lim_{N \rightarrow \infty} \frac{|\sr[N]|}{\log N} = \infty.\end{equation}
Then there is a function $\eta : \sr \rightarrow \{1, -\frac{1}{3}\}$ such that 
\begin{equation}\label{eq02} \lim_{N \rightarrow \infty} \frac{1}{|\sr[N]|}\sum_{n \in \sr[N]} \eta(s) a(s) = 0\end{equation} for all nilsequences $a$.
\end{proposition}

\begin{proof} The space of $C^{\infty}$-functions on $G/\Gamma$ is dense in the space of continuous functions; to approximate a continuous function by a smooth function, average with respect to a smooth kernel supported near the identity on $G$. It therefore suffices to verify \eqref{eq02} for $a(n) = \phi (g^n x)$ with $\phi \in C^{\infty}(G/\Gamma)$. Now we use the fact that there is a map 
\[ \mbox{Complexity} : \{ \mbox{smooth nilsequences} \} \rightarrow (0,\infty)\] and a function $M : (0,\infty) \times (0,1) \rightarrow (0,\infty)$ such that the set
\[ \{ a : \mbox{Complexity}(a) \leq C\}\] can be covered by $N^{M(C,\eps)}$ balls of radius $\eps$ in $\ell^{\infty}[N]$.

Results of this type were first observed by Frantzikinakis \cite[Proposition 6.2]{nikos-hardy}, and in fact Proposition \ref{lem0.1} and its proof are very closely related to \cite[Theorem 1.4]{nikos-hardy}. A discussion which gives what we need here is in the appendix of Altman \cite{altman} (note that $(g^n x_0)_{n \in \Z}$ is a particular example of a polynomial sequence as considered by Altman).

We will pick the values of $\eta(n)$ at random, choosing $\eta(n) = -\frac{1}{3}$ with probability $\frac{3}{4}$, and $\eta(n) = 1$ with probability $\frac{1}{4}$, these choices being independent for different values of $n \in \sr$. Then $\E \eta(n) = 0$. By well-known large deviation estimates (Hoeffding's inequality), for any fixed $1$-bounded functon $b$, and for any distinct $n_1,\dots, n_m$,
\begin{equation}\label{large-dev} \P( |\sum_{i = 1}^m \eta(n_i) b(n_i)| \geq t  ) \ll e^{-ct^2/m}, \end{equation} where $c > 0$ is absolute.

Let $\omega : \N \rightarrow (0,\infty)$ be some function tending to infinity, to be specified later.

For each $N$, let $E_N$ be the following event: for all $1$-bounded nilsequences $a$ of complexity $\leq \omega(N)$, 
\begin{equation}\label{ej} |\sum_{n \in \sr[N] } \eta(n) a(n)| \leq \frac{1}{\omega(N)} | \sr[N] |.\end{equation}
We estimate $\P(E_N)$ as follows. Pick some collection $\{ a_1,\dots, a_J\}$, $J \leq N^{M(\omega(N),1/2\omega(N))}$ of functions such that, for every $1$-bounded nilsequence~$a$ of complexity at most $\omega(N)$, there is some $a_i$ with $\Vert a - a_i \Vert_{\ell^{\infty}[N]} \leq 1/2\omega(N)$. Note that we do not need to assume that the $a_i$ are nilsequences (though this could be arranged if desired) and they are automatically $2$-bounded.

If we are not in $E_N$, there is some $a_i$ such that 
\begin{equation}\label{eji} |\sum_{n \in \sr[N] } \eta(n) a_i(n)| \geq \frac{1}{2 \omega(N)} | \sr[N] |.\end{equation}

By \eqref{large-dev}, the probability of \eqref{eji} happening, for some fixed $i$, is bounded above by $e^{-c' |\sr[N]|/\omega(N)^2}$ for some $c' > 0$. Summing over $i$, it follows that 
\[ \P( \neg E_N) \leq N^{M(\omega(N), 1/2\omega(N))} e^{-c' |\sr[N]|/\omega(N)^2}.\]
Choose $\omega$ (with $\omega(N) \rightarrow \infty)$ so that 
\[ \frac{|\sr[N]|}{\log N} > \frac{\omega(N)^2}{c'} \big(  10 + M(\omega(N), 1/2\omega(N))\big)\] for $N$ sufficiently large. (Here, of course, we have used the assumption on $\sr$). This then means that
\[ \P(\neg E_N) \leq N^{-10}\] for large $N$. In particular, $\sum_N \P(\neg E_N) < \infty$ which, by Borel-Cantelli, implies that almost surely only finitely many of the $\neg E_N$ occur. In particular, there is some particular choice of $\eta$ such that \eqref{ej} holds for all sufficiently large $N$.  Since every nilsequence has finite complexity, this implies the result. \end{proof}

\emph{Remark.} There is of course nothing special about $\{1, -\frac{1}{3}\}$; any set containing both positive and negative numbers would do.\vspace*{8pt}

To conclude this section, let us quickly sketch how one could extend Proposition \ref{lem0.1} to include the case where $a()$ is a bracket polynomial or a product of such (and hence not a nilsequence with a \emph{continuous} automorphic function $\phi$). Write $\chi_{\alpha, \beta}(n) := e(\alpha n\lfloor \beta n\rfloor)$. The key point is that the set of functions $\chi_{\alpha, \beta}(n)$, like the set of nilsequences of fixed complexity, has polynomially-bounded covering numbers in $\ell^{\infty}[N]$. 

To see why this is so, first note that $\chi_{\alpha, \beta}$ depends only on $\alpha \md{1}$, so we may assume $0 \leq \alpha < 1$. Next, replacing $\beta$ by $\beta +k$ for $k \in \Z$ has the effect of multiplying by a quadratic phase $e(\gamma n^2)$ (where $\gamma = \alpha k$). However, the set of all quadratic phases $e(\gamma n^2)$ is covered by $\ll_{\eps} N^2$ balls of radius $\eps$ in $\ell^{\infty}[N]$, since we may assume $0 \leq \gamma < 1$ and changing $\gamma$ by $\frac{\eps}{N^2}$ only changes $e(\gamma n^2)$ by $O(\eps)$, uniformly for $n \leq N$.

It therefore suffices to show that the covering numbers of the set $\Xi := \{ \chi_{\alpha, \beta} : 0 \leq \alpha, \beta < 1\}$ are polynomially bounded in $\ell^{\infty}[N]$. Now, restricted to $n \leq N$, there are only polynomially many functions $\lfloor \beta n\rfloor$ as $\beta$ ranges in $[0, 1)$. Indeed, the map $\beta \mapsto (\lfloor \beta n\rfloor)_{n \leq N}$ is only discontinuous at the points where $\beta n \in \Z$ for some $n \leq N$, of which there are no more than $N^2$ with $0 \leq \beta < 1$. Thus $\chi_{\alpha, \beta} = \chi_{\alpha, \beta'}$, with $\beta'$ varying in a set of size $N^2$. Changing $\alpha$ by $\frac{\eps}{N^2}$ only changes $\chi_{\alpha, \beta}(n)$ by $O(\eps)$, uniformly for $n \leq N$. Therefore the covering number of $\Xi$ in $\ell^{\infty}[N]$ is $\ll_{\eps} N^4$.

It follows immediately that, for fixed $C$, the set of functions of type $e(\sum_{i = 1}^k \alpha_i n [\beta_i n])$, where $k \leq C$, is covered by $N^{M(C,\eps)}$ balls of radius $\eps$ in $\ell^{\infty}[N]$. One could include various types of 1-step nilsequence or bracket polynomial and obtain a similar result. 

Bounds on covering numbers were all we needed to know about nilsequences, and the rest of the argument goes over verbatim.

\section{The heart of the construction}\label{sec2}

Define  $\psi : \Z \rightarrow \{-1,1\}$ to be the function with $\psi(0) = 1$, $\psi(1) = \psi(2) = -1$, and periodic mod $3$. The crucial property of this function we will use is the following, which is easily checked:
\begin{equation}\label{3ap-orthog} \psi(x) \psi(x+d)\psi(x+2d) = \psi(x)\end{equation} if $d \equiv 0 \md{3}$, and $1$ if $d \neq 0 \md{3}$.

Fix, once and for all, a sequence $M_1 < M_2 < \cdots$ be a sequence of positive integers such that 
\begin{enumerate}
\item Each $M_i$ is a multiple of $3$;
\item $\lim_{n \rightarrow \infty} k^{-2} \sum_{i=1}^k \log M_i = 0$;
\item $\prod_{i=1}^{\infty} (1 - \frac{3}{M_i}) = \gamma > 0$.
\end{enumerate}
For instance, one could take $M_i = 3i^2$.

Define
\[ \Omega_{k} := \{ (x_1,x_2,\dots) : 0 \leq x_i < M_i, x_{k+1} = x_{k+2} = \cdots = 0\}.\] 
Later on we will need the technical variant
\[ \tilde \Omega_{k} := \{ (x_1,x_2,\dots) : 0 \leq x_i < M_i - 3, x_{k+1} = x_{k+2} = \cdots = 0\}.\] 
Define also $\Sigma_k$ to consist of all sequences $(x_1,x_2,\dots)$ with precisely two nonzero entries $x_a, x_b$, both of which equal 1, and with $x_{k+1} = x_{k+2} = \cdots = 0$. Write
\[ \Omega := \bigcup_k \Omega_{k}, \quad \tilde\Omega := \bigcup_k \tilde\Omega_{k}, \quad \Sigma := \bigcup_k \Sigma_k.\]

We have a bijective map
\[ \beta :  \Omega \rightarrow \Z_{\geq 0}\] defined by 
\[ \beta(x_1,x_2,\dots) = x_1 + M_1 x_2 + M_1 M_2 x_3 + \cdots .\]

Let $\sr = \beta(\Sigma)$. Thus $\sr$ consists of the sums of two distinct elements of the sequence $\{1, M_1, M_1M_2, M_1M_2M_3,\dots\}$. We claim that $\sr$ satisfies the hypothesis \eqref{log-der} of Lemma \ref{lem0.1}, that is to say $\lim_{N \rightarrow \infty} \frac{|\sr[N]|}{\log N} = \infty$.

To see this, let $k$ be maximal so that $M_1 \cdots M_k \leq N/2$. Then $|\sr[N]| \geq \binom{k}{2}$, whilst $\log (N/2) \leq \sum_{i=1}^{k+1} \log M_i$. Therefore it is enough that 
\[ \lim_{k \rightarrow \infty} k^{-2}\sum_{i=1}^{k+1} \log M_i = 0,\] which follows immediately from assumption (2) above.

We now apply Lemma \ref{lem0.1} to get a function $\eta : \sr \rightarrow \{1, -\frac{1}{3}\}$ satisfying \eqref{eq02}.  Define
\[ \Sigma^+_{k} := \{x \in \Sigma_k : \eta(\beta(x)) = 1\}\quad \mbox{and} \quad \Sigma^-_{k} := \{ x \in \Sigma_k :  \eta(\beta(x)) = -\frac{1}{3}\}.\]
Thus $\Sigma_k = \Sigma^-_k\cup \Sigma^+_k$.

We introduce one more piece of notation. If $z \in \Sigma_k$ and if $x \in \Omega_{k}$ then we write
\[ \sigma_z(x ) := \sum_{i \in [k] : z_i = 0} x_i.\]

Now we come to the crucial definition. Let $k \in \N$. For $x \in \Omega_{k}$ define
\begin{equation}\label{f-def} f_k(\beta(x)) = \prod_{z \in \Sigma^k_-} \psi(\sigma_z(x)).\end{equation}
Note that $\beta(\Omega_{k}) = [0,N_k-1]$, where
\begin{equation}\label{Nk-def} N_k := M_1 \cdots M_k,\end{equation}
and so $f_k$ is a well-defined function on $[0, N_k-1]$, taking values in $\{-1,1\}$. Define also the technical variant
\begin{equation}\label{f-tilde-def} \tilde f_k(\beta(x)) := 1_{x \in \tilde \Omega_k} f_k(\beta(x)).\end{equation} Thus $\tilde f_k$ is defined on $[0,N_k - 1]$ and takes values in $\{-1,0,1\}$. Extend both $f_k$ and $\tilde f_k$ to functions on all of $\Z_{\geq 0}$ by defining $f_k(n) = \tilde f_k(n) = 0$ for $n \geq N_k$.

The following lemma is the heart of the argument. Here, recall that $\gamma > 0$ is just a positive constant (appearing in point (3) of the list of properties satisfied by the $M_i$).

\begin{lemma}\label{limit-cor}
For $d \in \Z_{\geq 0}$, write 
\[ S_k(d) := \frac{1}{N_k} \sum_{n \in [0,N_k-1]} \tilde f_k(n) f_k(n + d) f_k(n + 2d).\] 
Then for $d \in \sr$ we have $\lim_{k \rightarrow \infty} S_k(d) = \gamma \eta(d)$.
\end{lemma} 
\begin{proof}
Let $d \in \sr = \beta(\Sigma)$. For $k$ large enough, $d \in \beta(\Sigma_k)$, and we will assume this is so in what follows.

From the definition of $\tilde f_k$, we see that the sum over $n$ ranges over $n = \beta(x)$, $x \in \tilde\Omega_{k}$. Now for $n$ of this form and for $d = \beta(y)$, $y \in \Sigma_k$, we have $x+y, x+ 2y \in \Omega_{k}$ and moreover
\[ \beta(x + y) = \beta(x) + \beta(y) = n + d,\]
\[ \beta(x + 2y) = \beta(x) + 2 \beta(y) = n + 2d.\] Note that this ``lack of carries'' was precisely the reason for defining the set $\tilde \Omega_{k}$. It follows that 
\[ S_k(d) = \E_{x  \in \Omega_{k}}  \tilde f_k(\beta(x)) f_k(\beta(x + y)) f_k(\beta(x + 2y)),\] for $d = \beta(y)$, $y \in \Sigma_k$.
Substituting the definitions of $f_k, \tilde f_k$ (and noting that $\sigma_z$ is linear), we see that 
\[ S_k(d) = \E_{x \in \Omega_{k}} 1_{x \in \tilde \Omega_k} \prod_{z \in \Sigma^k_-} \psi(\sigma_z(x))\psi(\sigma_z(x) + \sigma_z(y))\psi(\sigma_z(x)+ 2\sigma_z(y)).\]
From \eqref{3ap-orthog} it follows that 
\[ S_k(d) = \E_{x \in \Omega_{k}} 1_{x \in \tilde \Omega_k} \prod_{z \in \Sigma^k_-: \sigma_z(y) \equiv 0 \mdsub{3}} \psi(\sigma_z(x)).\]
Now both $y$ and $z$ here are vectors with only two nonzero entries and so $\sigma_z(y)$ takes only the values $0,1,2$ with $\sigma_z(y) = 0$ iff $y = z$. Therefore
\begin{equation}\label{two-averages} S_k(d) = \left\{ \begin{array}{ll} \E_{x \in \Omega_{k}} 1_{x \in \tilde\Omega_k}  \psi(\sigma_y(x)) & \mbox{if $y \in \Sigma^k_-$}\\  \E_{x \in \Omega_{k}} 1_{x \in \tilde\Omega_{k}} &  \mbox{if $y \in \Sigma^k_+$.}\end{array}  \right.\end{equation}
The second expression is
\[ \E_{x \in \Omega_{k}} 1_{x \in \tilde\Omega_{k}}  = \frac{|\tilde \Omega_k|}{|\Omega_k|} = \prod_{i=1}^k (1 - \frac{3}{M_i}) \rightarrow \gamma\] as $k \rightarrow \infty$. The first expression in \eqref{two-averages} may be written explicitly as 
\begin{equation}\label{av-3} \frac{|\tilde \Omega_k|}{|\Omega_k|}\E_{x \in \tilde\Omega_k}  \psi(x_1 + \cdots + \hat{x}_i + \cdots + \hat{x}_j + \cdots + x_k),\end{equation}
where $y$ has nonzero coordinates at $i,j$ and the hat means that $\hat{x}_i$ does not appear in the sum. Note, however, that $\tilde\Omega_k$ is a box with sidelengths $M_i - 3$, each of which is a multiple of $3$. Therefore $x_1 + \cdots + \hat{x}_i + \cdots + \hat{x}_j + \cdots + x_k$ is uniformly distributed mod $3$, as $x$ ranges uniformly over~$\tilde\Omega_k$, and the average in \eqref{av-3} is 
\[ \frac{|\tilde \Omega_k|}{|\Omega_k|} \cdot (-\frac{1}{3}) = -\frac{1}{3} \prod_{i=1}^k (1 - \frac{3}{M_i})  \rightarrow -\frac{\gamma}{3}.\] 
This completes the proof.
\end{proof}

\section{Putting everything together}

Our final task is to build a measure-preserving system from the functions constructed in the last section. For this we will need a slight variant of the usual Furstenberg correspondence principle, proven in a very similar way. An essentially equivalent statement may be found, for instance, in \cite[Proposition 3.3]{nikos-correspondence}.

\begin{proposition}\label{prop1.1}
Let $A \subset \R$ be a finite set. Suppose that for each $k \in \N$ we have functions $f_{0,k}, \cdots , f_{r,k} : \Z_{\geq 0} \rightarrow A$, and that $(N_k)_{k=1}^{\infty}$ is an increasing sequence of positive integers. Then there is a measure-preserving system $(X, \mathscr{B}, \mu, T)$ and functions $F_0,F_1,\dots, F_r \in L^{\infty}(\mu)$ such that the following is true: if $(d_1,\cdots, d_r)$ is a tuple of distinct positive integers such that 
\[ S(d_1,\dots, d_r) := \lim_{k \rightarrow \infty} \frac{1}{N_k} \sum_{n \in [0,N_k-1]} f_{0,k}(n) f_{1,k}(n + d_1)\cdots f_{r,k}(n + d_r)\] exists, then
\[ S(d_1, \dots, d_r) = \int_X F_0 \cdot T^{d_1} F_1 \cdot T^{d_2} F_2 \cdots T^{d_r} F_ r d\mu.\]
\end{proposition}

We will apply this with the functions constructed in the last section, taking $r = 2$, $f_{0,k} := \tilde f_k$, $f_{1,k} = f_{2,k} = f_k$, and $N_k = M_1 \cdots M_k$ as before.

By Proposition \ref{prop1.1} and Lemma \ref{limit-cor}, there is a measure-preserving system $(X,\mathscr{B}, \mu, T)$ together with functions $F_0, F_1, F_2 \in L^{\infty}(\mu)$ such that, writing $C_{F_0, F_1, F_2}(d) := \int_X F_0 \cdot T^d F_1 \cdot T^{2d} F_2 d\mu$, we have 

\begin{equation}\label{c-prop} C_{F_0, F_1,F_2}(d) = \eta(d) \quad \mbox{for $d \in \sr$}.\end{equation}
(Note it is clearly possible to scale the $F_i$ to remove $\gamma$ factor appearing in Lemma \ref{limit-cor}.) We claim that it is impossible to write
\[ C_{F_0, F_1, F_2}(n) = a(n) + b(n)\] with $a$ an integral combination of $2$-step nilsequences and $\Vert b \Vert_{\infty} \leq \frac{1}{100}$. 
Suppose that this were possible. Then, from \eqref{c-prop} and the fact that $\eta$ takes values in $\{1, -\frac{1}{3}\}$, we would have $(a(d) + b(d)) \eta(d) \in \{\frac{1}{9}, 1\}$ for all $d \in \sr$. However, $|b(d) \eta(d)| \leq \frac{1}{100}$, and therefore

\begin{equation}\label{real} \Re (a(d) \eta(d) ) \geq \frac{1}{9} - \frac{1}{100} > \frac{1}{10}\end{equation} for all $d \in \sr$.

Suppose that

\[ a(n) = \int_M a_m(n) d\sigma(m).\] Here, $M$ is a compact metric space, $\sigma$ is a complex Borel measure of bounded variation  and the $a_m$ are nilsequences, with the map $m \mapsto a_m(n)$ being in $L^{\infty}(\sigma)$. 

 Then \eqref{real} implies that
\[ \big| \frac{1}{|\sr[N]|} \sum_{n \in \sr[N]} a(n) \eta(n) \big| \geq \frac{1}{10}.\]
On the other hand we have
\[
\big| \frac{1}{|\sr[N]|} \sum_{n \in \sr[N]} a(n) \eta(n) \big|  \leq \int_M \big| \frac{1}{|\sr[N]|} \sum_{n \in \sr[N]} a_m(n) \eta(n) \big| d|\sigma|
\] However, by the choice of $\eta$ (Lemma \ref{lem0.1}) we have
\[ \lim_{N \rightarrow \infty} \frac{1}{|\sr[N]|} \sum_{n \in \sr[N]} a_m(n) \eta(n) = 0\] for all $m$. Therefore, by the dominated convergence theorem,
\[ \lim_{N \rightarrow \infty} \int_M \big| \frac{1}{|\sr[N]|} \sum_{n \in \sr[N]} a_m(n) \eta(n) \big| d|\sigma| = 0.\]

Putting these statements together gives a contradiction, and this completes the proof of Theorem \ref{mainthm}.

\appendix

\section{Generalised nilsequences}\label{app-A}

In this appendix we explain why our example does not seem to give a negative solution to \cite[Problem 1]{nikos-open-problems}. That is, we explain why our example (or similar ones) do not seem to be able to rule out the possibility that $C_{F_0, F_1, F_2}(n)$ is an approximate integral combination of \emph{generalised} $2$-step nilsequences, in which the automorphic function $\phi$ is allowed to be merely Riemann-integrable. In fact, our examples agree with $1$-step generalised nilsequences on the crucial set $\sr$.

Recall that $\sr = \mathscr{A} \hat{+} \mathscr{A}$, where
\[ \mathscr{A} = \{ N_0, N_1, N_2,\dots\} \quad \mbox{and} \quad  N_i := \prod_{j \leq i} M_j\] (thus $N_0 = 1$, $N_1 = M_1$, $N_2 = M_1M_2$ and so on). Here, $\mathscr{A} \hat{+} \mathscr{A}$ means the restricted sumset of $\mathscr{A}$ with itself, that is to say the set of sums of two distinct elements of $\mathscr{A}$.

\begin{proposition}
There is $\theta \in \R/\Z$ such that the following is true. Let $\eta : \sr \rightarrow [-1,1]$ be any function. Then there is a Riemann-integrable function $\phi : \R/\Z \rightarrow [-1,1]$ such that $\phi(\theta n) = \eta(n)$ for all $n \in \sr$.
\end{proposition}
\begin{proof} Set $\theta := \sum_{i=1}^{\infty} \frac{1}{N_i}$. Since $M_1 < M_2 < \cdots$, we certainly have $M_j \geq j$. As a consequence, the usual proof that $e$ is irrational may be adapted easily to show that $\theta$ is irrational: if $\theta = \frac{p}{q}$ then $\alpha := \frac{M_1 \cdots M_q p}{q} \in \frac{1}{q}\Z$, but on the other hand the fractional part of $\alpha$ satisfies
\[ 0 < \{\alpha\} = \frac{1}{M_{q+1}} + \frac{1}{M_{q+1}M_{q+2}} + \cdots \leq \frac{1}{q+1} + \frac{1}{(q+1)(q+2)} + \dots < \frac{1}{q}.\]

Now define $\phi : \R/\Z\to [-1,1]$ as follows: $\phi(\theta n) = \eta(n)$ for all $n \in \sr$, and $\phi(x) = 0$ if $x \notin \theta \sr$. Since $\theta$ is irrational, this is a well-defined function. 

We claim that it is Riemann-integrable, with integral zero. It is enough to show that for every $\eps > 0$ there is some finite collection of intervals, of total length $< \eps$, whose union covers $\theta \sr$.

Note that for every $j$ we have
\begin{equation}\label{thet-approx} \Vert \theta N_j \Vert_{\R/\Z} = \frac{1}{M_{j+1}} + \frac{1}{M_{j+1}M_{j+2}} + \dots < \frac{1}{M_{j+1} - 1}.\end{equation}

Moreover, condition (3) in the definition of the $M_j$s implies that 
\begin{equation}\label{eq4} \limsup_{j \rightarrow \infty} \frac{M_j}{j} = \infty.\end{equation}
In particular we may choose $k$ so that  $\frac{1}{M_{k+1} - 1} < \frac{\eps}{10k}$, and by \eqref{thet-approx} it follows that 
\[ \Vert \theta N_j \Vert_{\R/\Z} < \frac{\eps}{10k} \quad \mbox{for $j \geq k$}.\] It follows that 
\[ \theta \mathscr{A} \subseteq \{ \theta N_0,\dots, \theta N_{k-1}\} \cup I,\] where $I = (-\eps/10k, \eps/10k) \subseteq \R/\Z$.
Therefore
\[ \theta \sr \subseteq \theta \mathscr{A} + \theta \mathscr{A} \subseteq \bigcup_{i,j < k} \{ \theta(N_i + N_j) \} \cup \bigcup_{i < k}( \theta N_i + I) \cup (I + I),\]
which makes it clear that $\theta \sr$ is contained in a finite union of intervals of length $< \eps$. 
\end{proof}


\begin{thebibliography}{99}

\bibitem{altman} D.~Altman, \emph{On Szemer\'edi's theorem with differences from a random set}, Acta. Arith. \textbf{195} (2020), no. 1, 97--108.


\bibitem{bhk} V.~Bergelson, B.~Host and B.~Kra, \emph{Multiple recurrence and nilsequences}, Invent. Math. \textbf{160} (2005), no. 2, 261--303.

\bibitem{bl-bracket} V.~Bergelson and A.~Leibman, \emph{Distribution of values of bounded generalised polynomials}, Acta. Math. \textbf{198} (2007), no. 2, 155-230.


\bibitem{jop-random-szem} J.~Bri\"et and S.~Gopi, \emph{Gaussian width bounds with applications to arithmetic progressions in random settings}, Int. Math. Res. Not. (IMRN), published online October 2018. 

\bibitem{jop-labib} J.~Bri\"{e}t and F.~Labib, in preparation.

\bibitem{nikos-hardy} N.~Frantzikinakis, \emph{Equidistribution of sparse sequences on nilmanifolds}, J. Anal. Math. \textbf{109} (2009), 353--395.


\bibitem{nikos-open-problems} N.~Frantzikinakis, \emph{Some open problems on multiple ergodic averages}, Bull. Hellenic Math. Soc. \textbf{60} (2016), 41--90.

\bibitem{nikos-correspondence} N.~Frantzikinakis, \emph{An averaged Chowla and Elliott conjecture along independent polynomials}, Int. Math. Res. Not. (IMRN) 2018, no. 12, 3721--3743.



\bibitem{yekhanin} S.~Yekhanin, \emph{Towards 3-query locally decodable codes of subexponential length}, J. ACM \textbf{55} (2008), no. 1, Art. 1, 16pp.




\end{thebibliography}
\end{document}